\documentclass[12pt,reqno]{amsart}
\usepackage[english]{babel} \usepackage{latexsym, amsfonts, amsmath, amssymb, amsthm, mathrsfs}
\usepackage{lmodern} \usepackage{color} \usepackage{mathpazo} \usepackage[pdftex]{hyperref}
\usepackage[top=3cm, bottom=3cm, left=3cm, right=3cm]{geometry}
\numberwithin{equation}{section}  \makeatletter\@addtoreset{equation}{section}
\newtheorem {theorem}{Theorem}[section]         \newtheorem {lemma}[theorem]{Lemma}     \newtheorem {definition}[theorem]{Definition}
     \newtheorem {remark}[theorem]{Remark}   
       \newtheorem {proposition}[theorem]{Proposition}
\newcommand{\C}{\mathbb C}    \newcommand{\R}{\mathbb R}   \newcommand{\Z}{\mathbb Z} 	
\newcommand{\V}{\mathbb V}
  \newcommand{\scal}[1]{\left<#1\right>}

\newcommand{\FHg}{ \mathcal{F}^{2,\nu, H}_{\Gamma, \chi} }

\newcommand{\FHgr}{ \mathcal{F}^{2,\nu,H}_{\Gamma_r, \chi} }
\newcommand{\Onugr}{\mathcal{O}^{\nu,H}_{\Gamma_r, \chi}}
\newcommand{\Fnugr}{\mathcal{F}^{2,\nu,H}_{\Gamma_r, \chi}}

\newcommand{\Fnugo}{\mathcal{F}^{2,\nu,H}_{\Gamma_0, \chi}}

\newcommand{\Fnugd}{\mathcal{F}^{2, \nu,H}_{\Gamma_{2g}, \chi}}

\newcommand{\Fnugu}{\mathcal{F}^{2, \nu,H}_{\Gamma_{1}, \chi}}

\newcommand{\psiNuU}{e^{\frac \nu 2 B(z,z) } }
\newcommand{\psiAlphaU}{e^{ 2\pi i \alpha z} }
\newcommand{\psiNuAlphaU}{e^{\frac \nu 2 B(z,z) + 2\pi i \alpha z} }
\newcommand{\psiNuAlphaUn}{e^{\frac \nu 2 B(z,z) + 2\pi i (\alpha+n) z} }
\newcommand{\psiNu}{\psi_\nu }
\newcommand{\psiAlpha}{\varphi_\alpha }
\newcommand{\psiNuAlpha}{\psi_{\nu,\alpha} }
\newcommand{\normnu}[1]{\left\Vert#1\right\Vert_{\nu,\Gamma,H}}

\newcommand{\normnur}[1]{\left\Vert#1\right\Vert_{\nu,\Gamma_r,H}}

\newcommand{\scalnuer}[1]{\left<#1\right>_{\nu,\Gamma_r}}
\newcommand{\mes}{d\lambda}
\newcommand{\nuw}{\widetilde{\nu}}
\newcommand{\WrR}{\mathbb{W}_\R^r}
\newcommand{\WrRpe}{(\mathbb{W}_\R^r)^{\perp_E}}

\newcommand{\WrC}{\mathbb{W}_\C^r}
\newcommand{\WrCp}{(\mathbb{W}_\C^r)^\perp}
\newcommand{\VgC}{\C^g}

\newcommand{\enk}{ e^{\alpha,\nu}_{n,k}}
\newcommand{\enkp}{ e^{\alpha,\nu}_{n',k'}}

\begin{document}
\date{\today}
\title[]{On holomorphic theta functions associated to rank $r$ isotropic discrete subgroups of a $g$-dimensional complex space}
\author{A. Ghanmi}              \email{ag@fsr.ac.ma}
\author{A. Intissar}            \email{intissar@fsr.ac.ma}
\author{M. Souid El Ainin}      \email{msouidelainin@yahoo.fr}
 \address{E.D.P. and Spectral Geometry,
          Laboratory of Analysis and Applications-URAC/03,
          Department of Mathematics, P.O. Box 1014,  Faculty of Sciences,
          Mohammed V University of Rabat, Morocco}

  \email{ag@fsr.ac.ma, intissar@fsr.ac.ma, msouidelainin@yahoo.fr}
\begin{abstract}
We are interested in the $L^2$-holomorphic automorphic functions on a $g$-dimensional complex space $\V^g_{\C}$
 endowed with a positive definite hermitian form and associated to isotropic discrete subgroups $\Gamma$ of rank $2\leq r \leq g$. 
 We show that they form an infinite reproducing kernel Hilbert space which looks like a tensor product of a theta Fock-Bargmann space 
 on $\V^{r}_{\C}=Span_{\C}(\Gamma)$ and the classical Fock-Bargmann space on $\V^{g-r}_{\C}$. 
 Moreover, we provide an explicit orthonormal basis using Fourier series and we give the expression 
 of its reproducing kernel function in terms of Riemann theta function of several variables with special characteristics.
\end{abstract}



\maketitle


\section{Introduction}
Let $\V=\V^g_{\C}$; $g\geq 1$, be a $g$-dimensional complex space with a positive definite hermitian form  $H(u,v)$. The real vector space underlying $\V$ carries a canonical
symplectic structure $(\V^{2g}_{\R}, E)$, when equipped with the nonsingluar alternating bilinear form $E(u,v):=\Im (H(u,v))$.
To a fixed real number $\nu>0$, given discrete subgroup $\Gamma$ of the additive group $(\V,+)$ and given mapping $\chi$ on $\Gamma$ with values in the unit circle of $\C$, we assign the functional space $\FHg(\V)$ of all holomorphic functions on $\V=\V^g_{\C}$ satisfying the functional equation
\begin{align}\label{FctEqIntr}
f(u+\gamma)= \chi(\gamma) e^{\nu H(u+\frac \gamma 2, \gamma)} f(u)
\end{align}
for every $ u\in \V$ and $\gamma\in\Gamma$, and such that
$$
\normnu{f}^2 :=\int_{\V_{\R}^{2g}/\Gamma } |f(u)|^2 e^{-\nu H(u,u)} \mes(u) <+\infty,
$$
where $\V_{\R}^{2g}/\Gamma $ denotes the orbital abelian group
of $\Gamma$ endowed with its Haar measure.
The cocompact case (i.e. $\Gamma$ is of maximal rank $r=2g$ or equivalently $\V^{2g}_\R/\Gamma$ being compact) 
is well studied in the literature \cite{Poincare16-54,Siegel1955,Baily1973} and the corresponding space $\Fnugd(\V)$ 
has a high interest not only on its own, but also in the light of
the remarkable implications for both pure mathematics and mathematical physics. It is closely connected to number 
theory and abelian varieties \cite{Shimura1968,Serre1973,BumpFriedbergHoffstein1996,PolishChuk2003}, representation 
theory \cite{GeifandPiateskii-Shapiro1963,Satake1971}, spectral analysis \cite{Iwaniec2002,GhIn2008,GhIn2012}, cryptography 
and coding theory \cite{Ritzenthaler2004,ShaskaWijesiri2008}, 
chaoticity of a shift operator \cite{InAb2013,InAb2015}  and quantum field theory \cite{Cartier1966,Fubini1991}.
Under the cocycle (Riemann-Dirac quantization (RDQ)) condition
 \begin{align}\label{RDQ-CC}
\chi(\gamma+\gamma')=\gamma(\gamma)\chi(\gamma')e^{i\nu E(\gamma, \gamma')}
 \end{align}
 for varying $ \gamma, \gamma'\in \Gamma=\Gamma_{2g}$, the space $\Fnugd(\V)$ of $(L^2,\Gamma,\chi)$-holomorphic 
 automorphic functions is of finite dimension which involves the volume of the complex torus $\V^g_{\C}/\Gamma_{2g}$. 
 Moreover, it appears as the null space of a special magnetic Laplacian $L_\nu$
acting on the whole space of $(L^2,\Gamma,\chi)$-automorphic functions (see \cite{GhIn2008} for more details).
 The spectral properties of $L_\nu$ when $\chi\equiv 1$ can be read off from the spectral analysis of the sub-Laplacian 
 $\mathcal{L}_\nu$ acting on the space of smooth $\widetilde{\Gamma}_{2g}$-periodic functions on the Heisenberg group $H^{2g+1}$,  
 where $\widetilde{\Gamma}_{2g}$ is a lattice subgroup of $H^{2g+1}$ whose projection onto $\V^g_\C$ is $\Gamma_{2g}$ 
 (see \cite{Folland2004} for details).

Looking for the non-cocompact case  (i.e. $\Gamma$ of rank $r$; $0\leq r<2g$) is therefore strongly motivated under a
lot of different viewpoints (see \cite{AbeKopfermann2001,Abe2003,Chan2013} for example and references therein).
 However, the explicit construction of the holomorphic automorphic functions using Fourier series is still far from 
 being understood in the general case of $0\leq r<2g$, except for $r=0$ and $r=1$.
Indeed, for $r=0$ we have $\Gamma_0=\{0\}$ and $\Lambda(\Gamma_0)$ is the whole $\V$. For $\chi(0)=1$, the space 
$\Fnugo(\V^g_\C)$ is nontrivial and can be identified to the usual Fock-Bargmann space consisting of all holomorphic 
functions, $f\in \mathcal{O}(\V^g_\C)$, on $\V^g_\C$ that are $e^{-\nu H(u,u)} \mes$-square integrable, 
\begin{align}\label{Bargmann}
 \mathcal{B}^{2, \nu}_H(\V^g_\C) := \mathcal{O}(\V^g_\C) \cap L^2(\V^g_\C; e^{-\nu H(u, u)} \mes).
 \end{align}
The rank one discrete subgroups has been considered and discussed recently by the authors in \cite{GhIn2012,Souid2015}.
 It is shown there that the corresponding functional space is nontrivial if and only if $\chi$ is a character.
Moreover, $\Fnugu(\V^g_\C)$ is an infinite dimensional reproducing kernel Hilbert space. Its concrete description,  
including the explicit construction of an orthonormal basis, and the explicit expression of its reproducing kernel, 
are investigated in \cite{GhIn2012} for $g=1$ and generalized to high dimension in \cite{Souid2015}.
However, for $g \geq 2$ and $2\leq r  \leq 2g-1$, it is so difficult to guarantees the existence of a nonzero holomorphic function on $\C^g$
satisfying the functional equation \eqref{FctEqIntr} and such that $\normnur{f}^2 < +\infty$, without additional assumption on $\Gamma_r$.

In the present paper, along the way of \cite{GhIn2012,Souid2015}, we extend the obtained results to high rank $1\leq r \leq g$ 
under special assumption on $\Gamma=\Gamma_r$. Mainely, we show that if the discrete subgroup $\Gamma_r$ is isotropic of rank 
$r$; $1\leq r \leq g$, with respect to the symplectic structure $(\V^{2g}_{\R}, E)$ (see Definition \ref{DefIso}), then $\FHgr(\V)$ 
is nonzero if and only if $\chi$ is a character. In this case, the space $\FHgr(\V)$ is an infinite reproducing kernel Hilbert space 
and looks like a tensor product of the theta Fock-Bargmann space on $\WrC=vect_{\C}(\Gamma_r)$ (\cite{GhIn2008}) and the classical 
Fock-Bargmann space on $
(\WrC)^{\perp_H}$.
Moreover, we provide an orthonormal basis and we explicit its reproducing kernel function in terms of the Riemann theta function of 
several variables with special characteristics. The main results for which is aimed this paper are summarized in Theorem \ref{Thmbasis} 
relevant to the description of the Hilbert space $\FHgr(\V^g_\C)$, including the construction of an orthonormal basis, and in Theorem 
\ref{ThmRKHS} relevant to giving the explicit expression of the reproducing kernel.
The crucial idea is that the complex vector space $\V=\V^g_\C$ can be endowed with a $\C$-basis whose first vectors are the generators 
of the (isotropic) discrete subgroup $\Gamma_r$. This does not work when $r > g$ for general discrete subgroup $\Gamma_r$.

The layout of the paper is as follows.
In Section 2,  we fix notations and review some elementary results on the isotropic discrete subgroups of $(\V,+)$.
The space $\FHgr(\V^g_\C)$ is introduced and studied in Section 3.
Our main results, Theorem \ref{Thmbasis} and Theorem \ref{ThmRKHS}, are proved in Section 4.


\section{Preliminaries on isotropic discrete subgroups of $(\C^g,+)$}
We consider the $g$-dimensional complex vector space $\V^g_\C$; $g\geq 1$, endowed with a positive definite hermitian form $H(u,v)$. 
The corresponding symplectic form is $E(u,v)=\Im(H(u,v))$,
where $\Im$ denotes the imaginary part of complex numbers, so that 
\begin{align} \label{HE}
H(u,v) = E(iu,v) + i E(u,v).
\end{align}
Let $\Gamma_r$ be a given discrete subgroup of the additive group $(\R^{2g};+)$ of rank $r$; $0\leq r \leq 2g$ (i.e., $\Gamma_r$ is 
a $\Z$-module of rank $r$). Then $\Gamma_r$ is spanned by some $\R$-linearly independent vectors $\omega_1, \cdots ,\omega_r\in \V^g_\C=\V^{2g}_\R$,
$\Gamma_r = \Z\omega_1+ \cdots + \Z\omega_r.$
We denote by $W^r_\R$ the corresponding $\R$-linear subspace
\begin{align}
W^r_\R = Span_\R\{\omega_1, \cdots ,\omega_r\} = Span_\R(\Gamma_r).
\end{align}

\begin{definition}\label{DefIso}
The discrete subgroup $\Gamma_r$ is said to be isotropic with respect to $E$ if the symplectic form $E$ vanishes on $\Gamma_r\times \Gamma_r$ 
and therefore on $W^r_\R\times W^r_\R$ by linearity.
\end{definition}
Notice that the rank one discrete subgroups of $(\V^{2g}_\R;+)$ are all isotropic. Note also that the following assertions
   \begin{enumerate}
     \item[(i)] $\Gamma_r$ is isotropic with respect to $E$,
     \item[(ii)] The $\R$-linear subspace $W^r_\R$ is contained in its symplectic complement
            \begin{align}
            \WrRpe = \{v\in \R^{2g}; \, E(\omega_j,v)=0, \, j=1, \cdots ,r\} ,
            \end{align}
     \item[(iii)] For every $j,k=1,\cdots,r$, we have  $ H(\omega_j, \omega_k)= H(\omega_k, \omega_j),$
   \end{enumerate}
are clearly equivalents. Moreover, we have

\begin{lemma}\label{lem-Iso1} Let  $\Gamma_r$ be an isotropic discrete subgroup of $(\R^{2g};+)$ of rank $r$. Then, we have
\begin{enumerate}
  \item[(i)]    $ 0 \leq r \leq g$.
  \item[(ii)]   $H(u,v)=E(iu,v)$ for every $u, v \in \WrR$.
  \item[(iii)]  $\WrR \cap i \WrR =\{0\}$.
  \item[(iv)]   $\WrC :=  \WrR \oplus i \WrR$ is a $\C$-linear subspace of $\V^g_\C$ spanned by the $\C$-linearly independent vectors 
  $\{\omega_1, \omega_2, \cdots, \omega_r \}$,  $ \WrC  =Span_\C\{\omega_1, \cdots ,\omega_r\}$. In particular,  $\dim_\C\WrC =r$.
\end{enumerate}
\end{lemma}

\begin{proof}
 Combination of the facts $\dim_\R\WrR=r$, $\WrR \subset \WrRpe$ and  $$\dim_\R\WrRpe + \dim_\R\WrR = \dim_\R\V^{2g}_\R=2g$$
(see \cite{.}) infer $(i)$.
$(ii)$ is an immediate consequence of \eqref{HE} combined with the fact that $E(u,v)=0$ for every $u,v\in \WrR$.
$(iii)$ follows easily form $(ii)$.
For $(iv)$, $ \WrC  =Span_\C\{\omega_1, \cdots ,\omega_r\}$ is clear and the $\C$-linear independence of $\{\omega_1, \omega_2, \cdots, \omega_r \}$ 
can be obtained using $(iii)$.
\end{proof}

Thus, under the assumption that $\Gamma_r$ is isotropic, the vector space $\V^g_\C$ can then be seen as direct sum, $\V^g_\C = \WrC \oplus \WrCp$,
 where $\WrCp$ denotes the orthogonal complement of $\WrC$ with respect to $H$. The subspace $\WrCp$  is then generated by some $\C$-linearly 
 independent vectors, $\omega_{r+1}, \cdots ,\omega_g \in \V^g_\C$. 
 Without lost of generality, we can assume that the $\omega_{r+1}, \cdots ,\omega_g$ are orthogonal with respect to the hermitian form $H$. That is
  $$
  H(\omega_j, \omega_k)= \delta_{jk}; \quad j,k=r+1,\cdots,g.
  $$
Accordingly, we can decompose $H$ on $\V^g_\C\times \V^g_\C$ as
\begin{align} \label{Hdecomp}
H(u,v)= \sum_{j,k=1}^r z_j\overline{w_k} H(\omega_j,\omega_k)+\sum_{j=r+1}^g z_j\overline{w_j},
\end{align}
for every $u=\sum_{j=1}^gz_j \omega_j, v=\sum_{k=1}^g w_k \omega_k\in \V^g_\C$.
Furthermore, since $H(\omega_j,\omega_k)= H(\omega_k,\omega_j)$ for every $1\leq j,k \leq r$, the form $(u,v)\longmapsto H(u,\overline{v})$, 
where $\overline{v}=\sum_{k=1}^g \overline{w_k} \omega_k$, induces a non-degenerate symmetric bilinear form on $\WrC\times \WrC$ by considering
 \begin{align} \label{Bform}
B(u,v)=H(u,\overline{v})=\sum_{j,k=1}^r z_j w_k H(\omega_j,\omega_k); \qquad u,v \in \WrC.
\end{align}
Whence the matrix $B=(H(\omega_j,\omega_k))_{1\leq j,k \leq r}$ is non-degenerate, real and symmetric.
Whenever $B$ denotes the form \eqref{Bform} on $\WrC\times \WrC$, $\widetilde{B}$ denotes its natural extension to $\V^g_\C\times \V^g_\C$, 
$\widetilde{B}(u,v)=H(u,\overline{v})$; $u,v\in \V^g_\C$.

We conclude this section by claiming that

\begin{lemma}\label{lem-Iso2}
 Let  $\Gamma_r$ be an isotropic discrete subgroup of $(\R^{2g};+)$ of rank $r$. Then, we have
\begin{enumerate}
  \item[(i)]   $H(u,v)=H(\overline{u},\overline{v})$ for every $u, v \in \WrC$.
  \item[(ii)]  $H(u,\gamma)=\widetilde{B}(u,\gamma)$ for every $u\in \V^g_\C$ and $\gamma\in \Gamma_r$.
  \item[(iii)] $\widetilde{B}(u+\gamma,u+\gamma)= \widetilde{B}(u,u) + 2 H\left(u+\frac{\gamma}{2},\gamma\right)$
               for every $u\in \V^g_\C$ and $\gamma\in \Gamma_r$.
\end{enumerate}
\end{lemma}

\begin{proof}
(i) follows by symmetry of $B$ on $\WrC\times \WrC$, while (ii) is immediate since $\overline{\gamma} =\gamma$.
 Finally, the bilinearity and the symmetry of $\widetilde{B}$ with  $(ii)$ yield $(iii)$;
$$\widetilde{B}(u+\gamma,u+\gamma)= \widetilde{B}(u,u) + 2\widetilde{B}\left(u+\frac{\gamma}2,\gamma\right)= \widetilde{B}(u,u) + 2H\left(u+\frac{\gamma}2,\gamma\right).$$
\end{proof}


\section{Statement of main results}
  For simplicity of exposition, we set $\V^g_\C=\VgC$ and $\V^{2g}_\R=\R^{2g}$.
For given fixed real number $\nu>0$, given discrete subgroup $\Gamma_r = \Z\omega_1+ \cdots + \Z\omega_r$; $0\leq r\leq 2g$, of $(\C^g,+)$
and given mapping $\chi$ on $\Gamma_r$ such that $|\chi(\gamma)|=1$, we consider the space $\Onugr(\C^g)$ of complex valued holomorphic 
functions on $\C^g$ displaying the functional equation
\begin{align}\label{FctEq1}
f(u+\gamma) = \chi(\gamma)  e^{\nu H\left(u+\frac{\gamma}2, \gamma\right) } f(u);\quad u\in \C^g,  \gamma\in\Gamma_r.
\end{align}
Notice that the function $|f(u)|^2 e^{-\nu H(u, u)}$ is $\Gamma_r$-periodic for every $f$ satisfying \eqref{FctEq1}. Therefore the quantity
 \begin{align}\label{IntrNorm1}
\normnur{f}^2 := \int_{\Lambda(\Gamma_r)} |f(u)|^2 e^{-\nu H(u, u)} \mes(u)
\end{align}
makes sense and is independent of the choice of $\Lambda(\Gamma_r)$, where $\mes$ denotes the Lebesgue measure on $\C^g$. 
Here $\Lambda(\Gamma_r)$ is a fundamental domain of $\Gamma_r$ representing in $\R^{2g}=\C^g$ the orbital group $\R^{2g}/\Gamma_r$ 
with respect to its quotient topology. It can be identified to $\Lambda(\Gamma_r)\equiv  ([0,1])^r\times \R^{2g-r}$ and therefore to
$$\Lambda(\Gamma_r) \equiv ([0,1]\times \R)^r \times \C^{g-r}$$
when $0\leq r \leq g$.
Thence, we perform the functional space  $\Fnugr(\C^g)$  of all holomorphic functions on $\C^g$ satisfying
\eqref{FctEq1} and such that $\normnur{f}^2$ is finite. That is
$$
\Fnugr (\C^g)= \Onugr(\C^g) \cap L^2(\Lambda(\Gamma_r); e^{-\nu H(u, u)} \mes) .
$$ 
The following result can be obtained in a similar way as in \cite{GhIn2008} (see also \cite{GhIn2012,Souid2015}).

\begin{proposition}\label{prop-trivialr}
Let $(\C^g,H,E)$ and $(\nu,\Gamma_r,\chi)$ be fixed as above. If the space $\Onugr(\C^g)$ (resp. $\Fnugr(\C^g)$) is non-trivial, 
then the triplet $(\nu,\Gamma_r,\chi)$ verify the cocycle condition
$$\chi(\gamma+\gamma')=\chi(\gamma)\chi(\gamma')e^{i\nu E(\gamma,\gamma')}; \quad \forall \gamma,\gamma'\in \Gamma_r .\eqno{(RDQ)} $$
\end{proposition}

\begin{remark}[Geometric interpretation]
The Riemann-Dirac quantization condition (RDQ) is equivalent to say that the family of multipliers
$  J^{\nu,\chi}_\gamma(u):=\chi(\gamma )e^{\nu H\left(u+\frac{\gamma}2,\gamma\right)}$, $\gamma\in \Gamma_r$, defines a line bundle over 
the quasi-torus $\VgC/\Gamma_r$ as the quotient of $\VgC\times\C$ with respect to the action of $\Gamma_r$ given by the mapping
   $\phi_\gamma(u;v) := (u+\gamma; J^{\nu,\chi}_\gamma(u)    v)$.
 Thus, $\Fnugr(\VgC)$ can be seen as the space of holomorphic sections of the above line bundle.
\end{remark}

\begin{remark}[Group theoretic viewpoint]
 Under the $(RDQ)$ condition, the map $\chi$ induces a group homomorphism from $(\Gamma_r,+)$ into the Heisenberg group 
 $N_E:=(\VgC\times U(1); \cdot_E)$ endowed with the law group $\cdot_E$ defined by
  $
  (u;\lambda) \cdot_E (v;\mu):= (u+v; \lambda\mu e^{i\nu E(u,v)}),
  $
  by considering the injection map $ i_E(\gamma):= (\gamma,\chi(\gamma))$.
\end{remark}

\begin{remark} If (RDQ) holds then necessary $\frac{\nu}{\pi}E(\gamma,\gamma')$ takes integers values on $\Gamma_r \times \Gamma_r$. Moreover,
 $\chi$ is a character if and only if $\nu E(\gamma,\gamma')$ belongs to $2\pi\Z$ for every $\gamma,\gamma'\in \Gamma_r$.
\end{remark}

In the sequel, we have to study the functional spaces  $\Onugr(\VgC)$ and $\Fnugr(\VgC)$ for the special class of 
isotropic discrete subgroups $\Gamma_r$ of $(\C^g,+)$, with $2 \leq r \leq g$. To this end, we need to fix additional notations.
We identify the elements of $\VgC$, $\WrC$ and $\WrCp$  with their complex coordinates with respect to the basis  $(\omega_1,  \cdots ,\omega_{g})$. 
Thus, we write $u= z_1\omega_1, \cdots,z_g\omega_g = (z_1, \cdots,z_g) \in \VgC$ as $u=(z,z_\perp)$ with $z=(z_1, \cdots,z_r)\in \C^r $ with respect 
to $\{\omega_1, \cdots,\omega_r\}$ and $z_\perp=(z_{r+1}, \cdots,z_g)\in \C^{g-r}$ with respect to $\{\omega_{r+1}, \cdots,\omega_g\}$.
Accordingly, we make use of the restriction $\widetilde{H}$ of $H$ to $\WrC$ and the usual scalar product $\scal{,}_{\C^{g-r}}$ on $(\WrC)^\perp$,
\begin{align} \label{HH}
\widetilde{H}(z,w)=  \sum_{j,k=1}^r z_j\overline{w_k}     H(\omega_j,\omega_k)  \quad \mbox{and} \quad  
\scal{z_\perp,w_\perp}_{\C^{g-r}}= \sum_{j=r+1}^g z_j\overline{w_j} , 
\end{align}
to rewrite $H(u,v)$ in \eqref{Hdecomp} for $u= (z,z_\perp)$ and $(w,w_\perp)$ as
\begin{align} \label{HHdecomp}
H(u,v)=  \widetilde{H}(z,w) +  \scal{z_\perp,w_\perp}_{\C^{g-r}} = B(z,\overline{w}) +  \scal{z_\perp,w_\perp}_{\C^{g-r}}.
\end{align} 
We also make use of the usual multi-index notations:
 $|n|=n_1+ \cdots + n_s$ and $n!=n_1! \cdots  n_s!$ for given multi-index $n=(n_1, \cdots , n_s) \in (\Z^+)^s$,  $z^n =z_1^{n_1}  \cdots  z_s^{n_s}$ and
 $z w =z_1w_1  \cdots  z_sw_s$ for given $z = (z_1, \cdots , z_s), w = (w_1, \cdots , w_s) \in \C^s$. 

Under the assumption that $\Gamma_r=\Z \omega_1+\Z \omega_2+\cdots+\Z \omega_r$,
  is isotropic, the $(RDQ)$ condition becomes equivalent to say that $\chi$ is a character. Then, it is 
  completely determined
\begin{align}\label{charact}
 \chi(\gamma)=\chi_\alpha(\gamma) 
 = e^{2\pi i \alpha m}
\end{align} 
for certain $\alpha= (\alpha_1, \cdots ,\alpha_r)\in \R^r$, with $\gamma =(m_1,m_2,\cdots,m_r)=m \in\Z^r$.
 Therefore, the functional equation  \eqref{FctEq1} reads
\begin{align}\label{EqFct3m}
f(z + m, z_\perp) = e^{\nu B\left(z + \frac{m}{2},m\right) + 2 \pi i \alpha m }f(z, z_\perp).
\end{align}
Notice that in \eqref{EqFct3m},  the $z_\perp$ can be seen as parameter. Thus we are dealing only with the $r$-complex variable $z=(z_1,\cdots, z_r)$. 
Nevertheless,  the condition of $L^2$-integrable will involves such parameter as we will see below.

The first main result of this paper is the following

\begin{theorem}\label{Thmbasis}
For given isotropic discrete subgroup $\Gamma_r$ and the character $\chi_\alpha$ given by \eqref{charact}, we have
\begin{enumerate}
  \item The functions $\enk(z,z_\perp)$ given by
 \begin{align}\label{fctbasis}
\enk(z,z_\perp):= e^{\frac{\nu}{2} B(z,z) + 2i\pi (\alpha + n) z}  z_{\perp}^{k},
 \end{align}
  for varying $n\in\Z^r$ and $ k\in(\Z^+)^{g-r}$, constitute an orthogonal basis in $\Fnugr(\VgC)$ with
 \begin{align}\label{normenk}
 \normnur{ e^{\alpha,{\nu} }_{n,k}}^2
   = \left(\frac{1}{\sqrt{\det B}}\right)  \left(\dfrac{\pi }{2\nu}\right)^{r/2} \left(\dfrac{\pi}{{{\nu}} }\right)^{g-r}
    \left(\frac{k!}{\nu^{|k|}}\right) e^{\frac{2\pi^2}{{{\nuw}} }(n+\alpha)B^{-1}(n+\alpha)}   .
    \end{align}
\item A function $f$  belongs to $\Fnugr(\VgC)$ if and only if it can be expanded in series as
 \begin{align}\label{expansion-r}
     f(z,z_\perp):= \sum\limits_{(n,k)\in\Z^r\times (\Z^+)^{g-r}}
     a_{n, k} \,  e^{\frac{\nu}{2} B(z,z) + 2i\pi (\alpha + n) z}  z_{\perp}^{k},
     \end{align}
   for some sequence of complex numbers $(a_{n,k})_{n;k}$ satisfying the growth condition:
   $$\left(\frac{1}{\sqrt{\det B}}\right) \left(\dfrac{\pi }{2\nu}\right)^{r/2} \left(\dfrac{\pi}{{{\nu}} }\right)^{g-r} 
   \sum\limits_{(n;k)\in\Z^r\times (\Z^+)^{g-r}} \left(\frac{k!}{\nu^{|k|}}\right) e^{\frac{2\pi^2}{{{\nuw}} }(n+\alpha)B^{-1}(n+\alpha)}
    |a_{n,k}|^2 <+\infty .$$
\end{enumerate}
\end{theorem}

\begin{remark}
The spaces $\Onugr(\VgC)$ and $\Fnugr(\VgC)$, associated to given isotropic discrete subgroup $\Gamma_r$ and character $\chi_\alpha$, are nonzero spaces.
\end{remark}

\begin{definition}
 $\Onugr(\VgC)$ (resp. $\Fnugr(\VgC)$) is called the space of holomorphic $(\Gamma_r,\chi_\alpha)$-theta (resp. $(L^2,\Gamma_r,\chi_\alpha)$-theta) 
 functions associated to the isotropic discrete subgroup $\Gamma_r$. $\Fnugr(\VgC)$ is also called $(\Gamma_r,\chi_\alpha)$-theta Fock-Bargmann space.
\end{definition}

Now, for every $z\in \C^r$ and $F \in \C^{r\times r}$ such that $F$ is symmetric and its imaginary part $Im(F)$ is strictly positive definite, we define
 the Riemann theta function  
 with characteristics $\alpha,\beta\in \R^r$ \cite{Riemann1857,Prym1882,Wirtinge1895,Mumford83} by
\begin{align}\label{defTheta}
 \Theta_{\alpha,\beta} (z \big |  F)
 = \sum\limits_{n\in\Z^r}
   e^{2i\pi \left\{ \frac 12 (\alpha+n)F(\alpha+n) + (\alpha+n)(z+\beta)\right\} }.
\end{align}
 The positive definiteness of $Im(F)$ guarantees the convergence of \eqref{defTheta}, for all values of $z\in \C^r$.
Then, we assert

\begin{theorem}\label{ThmRKHS}
$\Fnugr(\VgC)$ is a reproducing kernel Hilbert space. Its reproducing kernel $K(u,v)$ is given in terms of the theta function $\Theta_{\alpha,0}$. 
More explicitly, for every $u=(z,z_\perp)$ and $v=(w,w_\perp)$ in $\VgC$, we have
   \begin{align*}
 K(u,v) =
 \sqrt{\det B} \left(\dfrac{2\nu}{\pi}\right)^{r/2} \left(\dfrac{\nu}{\pi}\right)^{g-r} e^{\frac{\nu}{2}\left( B(z,z) + \overline{B(w,w)} \right) }
\Theta_{\alpha,0}\left( z-\overline{w} \bigg | \frac{2\pi i}{\nu} B^{-1} \right)   e^{\nu\scal{z_\perp,w_\perp}_{\C^{g-r}} }.
\end{align*}
\end{theorem}


\section{Proofs of main results}


To prove Theorem \ref{Thmbasis}, we establish first some needed intermediary results.
We begin with the following

\begin{lemma} Set
 $  \psiNu(u) = \psiNuU$ and $\psiAlpha(u):= \psiAlphaU $  for every $u=(z,z_\perp)\in \VgC$.
 Then, the holomorphic functions $\psiNu$ and $\psiAlpha$ satisfy the functional equations
  \begin{align}\label{fcteqnu}
 \psiNu(u+\gamma) = e^{\nu H(u+\frac{\gamma}{2},\gamma)} \psiNu(u)   \quad \mbox{and} \quad
 \psiAlpha(u+\gamma)= \chi_\alpha(\gamma) \psiAlpha(u),
 \end{align}
 and the nonzero holomorphic function  $\psiNuAlpha(u) :=\psiNu(u) \psiAlpha(u)$ belongs to $\Onugr(\VgC)$.
\end{lemma}

\begin{proof} The first functional equation in \eqref{fcteqnu} follows making use of $(iii)$ of Lemma \ref{lem-Iso2}. 
The second one is immediate keeping in mind that $ \chi_\alpha(\gamma) = e^{2\pi i \alpha m}$, with $\gamma =(m_1,m_2,\cdots,m_r)=m$.
 Wherefor, the holomorphic function  $\psiNuAlpha(u) :=\psiNu(u) \psiAlpha(u)$ satisfies
   \begin{align}\label{fcteqnualpha}
 \psiNuAlpha(u+\gamma)= \chi_\alpha(\gamma) e^{\nu H(u+\frac{\gamma}{2},\gamma)} \psiNuAlpha(u).
 \end{align}
\end{proof}

\begin{proposition} \label{prop-description}
 The following assertions are equivalents
\begin{enumerate}
  \item[(i)] $f \in \Onugr(\VgC)$.
  \item[(ii)] There exists a holomorphic $\Gamma_r$-periodic function $f^*$ such that
         \begin{align}\label{fctGperiodic}
              f(u) = \psiNuAlpha(u) f^*(u) = \psiNuAlphaU f^*(u).
         \end{align}
  \item[(iii)] There exists a sequence of  complex numbers $(a_{n,k})_{(n,k)\in\Z^r \times (\Z^+)^{g-r}}$ such that $f$ can be expanded in series as
   \begin{align}\label{expansionzr}
     f(z,z_\perp):= \sum\limits_{(n,k)\in\Z^r \times (\Z^+)^{g-r}} a_{n,k} \psiNuAlphaUn z_\perp^{k} .
     \end{align}
\end{enumerate}
\end{proposition}

\begin{proof}
The equivalence $(i) \Longleftrightarrow (ii)$ follows immediately making use of \eqref{fcteqnualpha},
$$\chi_\alpha(\gamma) e^{\nu H(u+\frac{\gamma}{2},\gamma)} =\frac{\psiNuAlpha(u+\gamma)}{\psiNuAlpha(u)}.$$
Indeed, the functional equation
$  f(u+\gamma) = \chi_\alpha(\gamma) e^{\nu H(u+\frac{\gamma}{2},\gamma)} f(u)$
becomes equivalent to
$$  f^{*}(u+\gamma) =  \frac{f(u+\gamma)}{\psiNuAlpha(u+\gamma)} = \frac{f(u)}{\psiNuAlpha(u)} = f^{*}(u).$$
Whence $f(u)= \psiNuAlpha(u) f^{*}(u)$ with $f^{*}$ is a holomorphic $\Gamma_r$-periodic function. 

The proof of $(ii) \Longleftrightarrow (iii)$ lies essentially on the fact that any holomorphic function in $(z,z_\perp)$ and
$\Gamma_r$-periodic in $z$ can be expanded as 
$$f^{*}(z,z_\perp) =  \sum\limits_{n\in\Z^r} \psi_n(z_\perp) e^{ 2i\pi n z}.$$
The series converges absolutely and uniformly on compact subsets of $\C^r \times\{z_\perp\}$ for every fixed $z_\perp\in\C^{g-r}$, 
while the Fourier coefficients
$\psi_n(z_\perp)$; $n\in\Z^r$, are holomorphic in $z_\perp$. Accordingly, they can be written as
       $$\psi_n(z_\perp)=\sum\limits_{k\in(\Z^+)^{g-r}} a_{n,k}z_{\perp}^k  $$
for some sequence of  complex numbers $(a_{n,k})_{k\in(\Z^+)^{g-r}}$. 
This gives rise to \eqref{expansionzr}.
\end{proof}

Consequently, the holomorphic functions
 \begin{align}\label{fctbasis}
\enk(z,z_\perp):= \psiNuAlphaUn z_\perp^{k} , 
 \end{align}
  for varying $n\in\Z^r$ and $ k\in(\Z^+)^{g-r}$, belong to $\Onugr(\VgC)$. Moreover, we assert the following

\begin{proposition}\label{prop-orthog}
  The functions $\enk(z,z_\perp)$ in \eqref{fctbasis} are orthogonal in $\Fnugr(\VgC)$, with
       \begin{align}\label{norm-def3r}
 \normnur{\enk}^2 = \left(\frac{1}{\sqrt{\det B}}\right) \left(\dfrac{\pi}{2\nu}\right)^{r/2} \left(\dfrac{\pi}{\nu}\right)^{g-r} 
  \left(\dfrac{k!}{\nu^{|k|}}\right)  e^{\frac{2\pi^2}{\nu}(n+\alpha)B^{-1} (\alpha+n)}.
     \end{align}
\end{proposition}

\begin{proof} From \eqref{HHdecomp} and making use of the Fubini's theorem, we can write
 \begin{align*}
\scalnuer{\enk, \enkp}
&= \left(\int_{([0,1]\times \R)^r} e^{\frac{{{\nu}} }{2} (B(z,z) + B(\overline{z},\overline{z}) -2 \widetilde{H}(z,z) )
 + 2i\pi \alpha(z - \overline{z}) + 2i\pi (n z - n' \overline{z} ) }   \mes(z)\right)
   \\ &\times   \left(\int_{\C^{g-r}}  z_{\perp}^{k} \overline{z_{\perp}}^{k'} e^{-\nu |z_\perp|^2_{\C^{g-r}} }\mes(z_\perp) \right)
\end{align*}
The second integral in the right hand-side of the last identity is well known. It can be obtained by means of Fubini's theorem and polar coordinates. 
Explicitly, we have
$$  \int_{\C^{g-r}}  z_{\perp}^{k} \overline{z_{\perp}}^{k'} e^{-\nu|z_\perp|^2_{\C^{g-r}}} \mes(z_\perp)
= \delta_{k,k'}  \left(\dfrac{\pi}{\nu}\right)^{g-r}  \left(\dfrac{k!}{\nu^{|k|}}\right).$$
For $z= x+iy$ with $x,y\in \R^r$, we have $n z - n' \overline{z} = (n  - n') x + i (n+n') y $.
Note also that since $\widetilde{H}(z,z) = B(z,\overline{z})$, it follows
 $$B(z,z) + B(\overline{z},\overline{z}) -2 \widetilde{H}(z,z)
 = B(z-\overline{z}, z - \overline{z}) = - 4 B(y, y).$$
Whence, we get
 \begin{align}
\scalnuer{\enk,\enk}
&= \delta_{k,k'} \left(\int_{([0,1])^r} e^{2i\pi  (n  - n') x  } dx \right)  \nonumber
 \\& \times   \left(\int_{\R^r} e^{ - 2\nu \widetilde{H}(y, y)  - 2\pi (2\alpha+ n + n') y }   dy\right) \left(\dfrac{\pi}{\nu}\right)^{g-r}  
 \left(\dfrac{k!}{\nu^{|k|}}\right)      \nonumber
   \\ &= \delta_{n,n'} \delta_{k,k'}  \left(\int_{\R^r} e^{ - 2\nu B(y, y)  - 4\pi (\alpha+n)y}   dy\right)
     \left(\dfrac{\pi}{\nu}\right)^{g-r}  \left(\dfrac{k!}{\nu^{|k|}}\right) \label{beforeGauss}
\end{align}
thanks to the well-known fact $ \int_{[0,1]^r} e^{2i\pi  (n  - n') x  } dx =  \delta_{n,n'}$.
One recognizes in the last integral the Gaussian integral \cite[p. 256]{Folland2004}
 \begin{align} \label{gaussIntegral}
  \int_{\R^r} e^{-a y A y + by} dy = \left(\frac{1}{\sqrt{\det A}}\right)\left(\dfrac{\pi}{a}\right)^{r/2}  e^{\frac{1}{4 a} b A^{-1} b },
  \end{align}
 where we have the limitation for $a>0$, $b\in \C^r$ and $A\in \C^{r\times r}$ a symmetric $r$-matrix whose $\Re(A)$ is positive definite. 
 Thus, by inserting \eqref{gaussIntegral}, for $a= 2\nu$, $b=- 4\pi (\alpha+n)$ and $A=B$ in \eqref{beforeGauss},  we get
 \begin{align*}
\scalnuer{\enk,\enk}
= \delta_{n,n'} \delta_{k,k'}\left(\frac{1}{\sqrt{\det B}}\right)
\left(\dfrac{\pi}{2\nu}\right)^{r/2} \left(\dfrac{\pi}{\nu}\right)^{g-r}  \left(\dfrac{k!}{\nu^{|k|}}\right) 
e^{\frac{2\pi^2}{\nu}(n+\alpha)B^{-1} (\alpha+n)}.
\end{align*}
\end{proof}

Now, we are able to prove Theorem \ref{Thmbasis}.

\begin{proof}[Proof of Theorem \ref{Thmbasis}]
In view of Propositions \ref{prop-description} and \ref{prop-orthog}, the functions $\enk$; 
are generators of $\Fnugr(\VgC)$ and pairewise orthogonal. We only need to prove that any $f \in \Fnugr(\VgC)$ such that 
$ \scalnuer{f,\enk} = 0$, for all $(n,k)$,  is identically zero. This follows since the expansion  series of $f$, given through \eqref{expansionzr},
converges uniformly on the compact sets, including $K_R:= ([0,1]\times[-R,R])^r\times B_R^{g-r}$ that recovers $\Lambda(\Gamma_r)$, 
 where $R>0$ and $B_R^{g-r}:=\{ z_\perp \in \C^{g-r}; \scal{z_\perp,z_\perp}_{\C^{g-r}} \leq R^2\}$. 
 Thus, one can exchange the series with the integration over subsets of $K_R$. Precisely, we have
\begin{align*}
\scalnuer{f,\enk} &=  \int_{\Lambda(\Gamma_r)} \left(\sum\limits_{(n',k')\in\Z^r \times (\Z^+)^{g-r}} a_{n',k'} \enkp(u) \right)
\overline{\enk(u)} e^{-\nu H(u,u)} \mes(u) \\
&= \lim_{R\to +\infty} \int_{K_R\cap \Lambda(\Gamma_r)} \left(\sum\limits_{(n',k')\in\Z^r \times (\Z^+)^{g-r}} a_{n',k'} \enkp(u) \right)
\overline{\enk(u)} e^{-\nu H(u,u)} \mes(u) \\
&= \lim_{R\to +\infty} \sum\limits_{(n',k')\in\Z^r \times (\Z^+)^{g-r}} a_{n',k'}
\int_{K_R}  \enkp(u) \overline{\enk(u)} e^{-\nu H(u,u)} \mes(u) \\
&= \lim_{R\to +\infty} a_{n,k}
\int_{K_R}  |\enk(u)|^2  e^{-\nu H(u,u)} \mes(u)\\
&= a_{n,k} \normnur{\enk}^2.
\end{align*}
Therefore, $a_{n,k} = 0$ for all $(n,k)$. This proves $(i)$ of Theorem \ref{Thmbasis}.

To prove $(ii)$, we note first that any $f \in \Onugr(\VgC)$ expands as
   \begin{align*}
     f(z,z_\perp)= \sum\limits_{(n,k)\in\Z^r \times (\Z^+)^{g-r}} a_{n,k}  \enk(z,z_\perp)
     \end{align*}
     for some sequence of complex numbers $(a_{n,k})_{(n,k)\in\Z^r \times (\Z^+)^{g-r}}$ (Proposition \ref{prop-description}).
Whence, $f$ belongs to $\Fnugr(\VgC)$ if and only if $\normnur{f}<+\infty$. In the other hand, we have
\begin{align}
\normnur{f}^2 &=  \sum\limits_{(n,k),(n',k')\in\Z^r \times (\Z^+)^{g-r}} a_{n,k}\overline{a_{n,k}} \scalnuer{\enk, \enkp} \nonumber\\
 &= \sum\limits_{(n,k)\in\Z^r \times (\Z^+)^{g-r}} |a_{n,k}|^2  \normnur{\enk}^2, \label{normfproof}
      \end{align}
 where $\normnur{\enk}$ is explicitly given by \eqref{normenk}.
\end{proof}

Now, in order to prove Theorem \ref{ThmRKHS}, we establish first the following

\begin{lemma}\label{LemEvaluation}
Define
\begin{align}\label{Kzz}
\widetilde{K}(z,z_\perp) = \sqrt{\det B } \left(\dfrac{2\nu}{\pi}\right)^{r/2} \left(\dfrac{\nu}{\pi}\right)^{g-r}
e^{\frac{\nu}{2}\left( B(z,z) + \overline{B(z,z)} \right) }
\Theta_{\alpha,0}\left( z-\overline{z} \bigg | \frac{2\pi i}{\nu} B^{-1} \right) e^{\nu |z_{\perp}|^{2}_{\C^{g-r} } }
\end{align}
for every $(z,z_\perp)\in \VgC$. Then, the following estimation
\begin{align}\label{Estimation}
|f(z,z_\perp)|  \leq  \sqrt{\widetilde{K}(z,z_\perp)}  \normnur{f}
\end{align}
holds for every $f\in\Fnugr(\VgC)$ and every $(z,z_\perp)\in \VgC $.
Moreover, for every compact set $K \subset \VgC$, there exists a constant $C_{K} \geq 0 $ such that
\begin{align}\label{EstimationK}
|f(z,z_\perp)|\leq C_K \normnur{f} \quad\quad;(z,z_\perp)\in K .
\end{align}
\end{lemma}

\begin{proof}
From the expansion series \eqref{expansionzr} of a given $f\in \Fnugr(\VgC)$, we get
\begin{align*}
 |f(z,z_\perp)| & \leq  \sum\limits_{(n,k)\in\Z^r\times (\Z^+)^{g-r}} |a_{n, k}| | \enk(z,z_\perp) | \\
                & \leq  \sum\limits_{(n,k)\in\Z^r\times (\Z^+)^{g-r}}\left( |a_{n, k}| \normnur{\enk} \right)
                \left(\frac{| \enk(z,z_\perp) |}{\normnur{\enk}}\right).
\end{align*}
Now, by means of the Cauchy-Schwarz inequality, it follows
\begin{align*} 
 |f(z,z_\perp)|  \leq  \left(\sum\limits_{(n,k)\in\Z^r\times (\Z^+)^{g-r}} |a_{n, k}|^2 \normnur{\enk}^2 \right)^{\frac{1}{2}}
  \left( \sum\limits_{(n,k)\in\Z^r\times (\Z^+)^{g-r}} \frac{| \enk(z,z_\perp) |^2}{\normnur{\enk}^2}\right)^{\frac{1}{2}}.
 \end{align*}
 The first term in the right hand side of the previous inequality is exactly the norm of $f$ (see \eqref{normfproof}). 
 The second one can be written in terms of the Riemann theta function $ \Theta_{\alpha,\beta} (z \big |  F)$ defined in 
 \eqref{defTheta} with $F=({2\pi i}/{\nu}) B^{-1}$ and $Im(F)=Re\left(({2\pi}/{\nu}) B^{-1}\right) >0$.
Indeed, from \eqref{fctbasis} and \eqref{norm-def3r}, we get
 \begin{align*}
  \sum_{(n,k)\in\Z^r\times (\Z^+)^{g-r}}   \frac{| \enk(z,z_\perp) |^2}{\normnur{\enk}^2}
     = \widetilde{K}(z,z_\perp) ,
\end{align*}
 where $\widetilde{K}(z,z_\perp)$ is as defined in \eqref{Kzz}.
Finally, we get $ |f(z,z_\perp)|  \leq  \sqrt{\widetilde{K}(z,z_\perp)}  \normnur{f}$.
The estimation \eqref{Estimation} follows since the involved function $\widetilde{K}(z,z_\perp)$ is continuous on $\VgC$, 
and therefore bounded on any compact set $K \subset \VgC$. We have $|f(z,z_\perp)|\leq C_K \normnur{f}$; $(z,z_\perp)\in K$, for certain constant $C_K$.
\end{proof}

With this we can handle the proof of Theorem \ref{ThmRKHS}.

\begin{proof}[Proof of Theorem \ref{ThmRKHS}]
The proof of $\Fnugr(\VgC)$ being a Hilbert space is classical, thanks to \eqref{EqFct3m} in Lemma \ref{LemEvaluation}.
 For completeness, we offer here a proof of it. Let  $(f_p)_p $ be a Cauchy sequence in $\Fnugr(\VgC)$.  Then, by means of 
 \eqref{EstimationK} for $f_p - f_q\in \Fnugr(\VgC)$, for every compact set $K \subset \VgC $  there exists certain constant $ C_K$ such that
$$ | f_p(z,z_\perp)- f_q (z,z_\perp)| \leq C_K\normnur{f_p- f_q } ,$$
for  any $(z,z_\perp)\in K $.
Therefore, the sequence $(f_p)_p$ of entire functions is uniformly
Cauchy on compact subsets, and then is uniformly convergent to an holomorphic function $f$ on the
whole $\VgC$. Furthermore, the limit function $f$ satisfies the same functional equation \eqref{EqFct3m} as $f_p$. 
To conclude, we need only to prove that $\normnur{f}< +\infty $.
 For this, notice that $(f_p)_p$ is also a Cauchy sequence in the Hilbert space 
 $ L^{2,\nu}_H(\Lambda(\Gamma_r)) := L^2(\Lambda(\Gamma_r); e^{- \nu H(u,u)} \mes(u))$ and whence
  $(f_p)_p$ converges to a function $\phi_{\Lambda(\Gamma_r)}\in L^{2,\nu}_H(\Lambda(\Gamma_r))  $ 
  in the norm of $L^{2,\nu}_H(\Lambda(\Gamma_r))$. Thus, there exists a subsequence $( f_{p_k} )_k$ of $(f_p)_p$ converging to 
  $\phi_{\Lambda(\Gamma_r)}$ almost everywhere on $\Lambda(\Gamma_r)$. Thence,
$f |\Lambda(\Gamma_r) = \phi_{\Lambda(\Gamma_r)}\in L^{2,\nu}_H(\Lambda(\Gamma_r)) $ almost everywhere on $\Lambda(\Gamma_r)$,  
and therefore $\normnur{f} = \normnur{\phi_{\Lambda(\Gamma_r)}} <+\infty$.

Now, according to Lemma \ref{LemEvaluation}, the point evaluation functional is continuous,  in $\Fnugr(\VgC)$ at the points of $\VgC$,
so that $\Fnugr(\VgC)$ is a reproducing kernel Hilbert space by Riesz representation theorem. Its reproducing kernel function can be determined explicitly.
Indeed, starting from the fact that
$$
K(u,v) =   \sum_{(n,k)\in\Z^r\times (\Z^+)^{g-r}}   \frac{\enk(z,z_\perp) \overline{ \enk(w,w_\perp)} }{\normnur{\enk}^2}
$$
for every $u=(z,z_\perp); v= (w,w_\perp)\in \VgC$, we obtain
\begin{align*}
K(u,v) &= \sqrt{\det B } \left(\dfrac{2\nu}{\pi}\right)^{r/2} \left(\dfrac{\nu}{\pi}\right)^{g-r}
e^{\frac{\nu}{2}\left( B(z,z) + \overline{B(w,w)} \right) }
 \\& \left(\sum\limits_{n\in\Z^r} e^{-\frac{2\pi^2}{\nu}(\alpha+n)B^{-1} (\alpha+n) + 2i\pi (\alpha+n)(z-\overline{w}) }\right)
 \left( \sum\limits_{k\in (\Z^+)^{g-r}}  \dfrac{\nu^{|k|}}{k!}z_{\perp}^k\overline{w_\perp^k} \right)  \\
 &= \sqrt{\det B} \left(\dfrac{2\nu}{\pi}\right)^{r/2} \left(\dfrac{\nu}{\pi}\right)^{g-r} e^{\frac{\nu}{2}\left( B(z,z) + \overline{B(w,w)} \right) }
\Theta_{\alpha,0}\left( z-\overline{w} \bigg | \frac{2\pi i}{\nu} B^{-1} \right)   e^{\nu\scal{z_\perp,w_\perp}_{\C^{g-r}} }   .
\end{align*}
\end{proof}



\begin{thebibliography}{99}

\bibitem{Abe2003} Abe Y.,
        {\it Construction of automorphic forms for ample factors of quasi-abelian varieties}.
        Kyushu J. Math. 57 (2003), no. 1, 51--85

\bibitem{AbeKopfermann2001} Abe Y., Kopfermann K.,
        {\it Toroidal groups. Line bundles, cohomology and quasi-abelian varieties}.
         Lecture Notes in Mathematics, 1759. Springer-Verlag, Berlin, 2001.

\bibitem{Baily1973} Baily jr. 	W.L.,
           {\it Introductory lectures on automorphic forms}.
           Iwanami Shoten $\&$ Princeton Univ. Press (1973)

\bibitem{Bargmann1961} Bargman V.,
       {\it On a Hilbert space of analytic functions and an associated integral transform}.
       Comm. Pure Appl. Math. 14 (1961) 187-214.

\bibitem{BumpFriedbergHoffstein1996} Bump D., Friedberg S., Hoffstein J.,
        {\it On some applications of automorphic forms to number theory}.
        Bull. Amer. Math. Soc. (N.S.) 33 (1996), no. 2, 157--175.

\bibitem{Cartier1966} Cartier P.,
       {\it Quantum mechanical commutation relations and theta functions. Algebraic Groups and Discontinuous Subgroups}.
       pp. 361-383; Proc. Sympos. Pure Math. 9,  Amer. Math. Soc.,  Providence (1966). 

\bibitem{Chan2013} Chan, Tsz On Mario.
        {\it The index theorem for quasi-tori}.
        J. Math. Pures Appl. (9) 100 (2013), no. 5, 719--747

\bibitem{Fubini1991} Fubini S.,
        {\it Finite Euclidean magnetic group and theta functions},
        Int. J. Mod. Phys. A 7 (1992) 4671.

\bibitem{Folland2004} Folland G.B., 
          {\it Compact Heisenberg manifolds as CR manifolds}.
           J. Geom. Anal. 14 (2004),  no. 3,  521-–532.

\bibitem{GeifandPiateskii-Shapiro1963}   Geifand I., Piateskii-Shapiro I.,
           {\it Theory of representations and theory of automorphic functions},
           Amer. Math. Soc. Transl. (2) 26 (1963), 173-200.

\bibitem{GhIn2008} Ghanmi A., Intissar A.,
        {\it Landau automorphic functions on $\C^n$ of magnitude $\nu$}.
        J. Math. Phys. 49 (2008), no. 8, 083503, 20 pp.

\bibitem{GhIn2012} Ghanmi A., Intissar A.,
         {\it Construction of concrete orthonormal basis for $(L^2,\Gamma,\chi)$-theta functions associated to discrete subgroups of rank one in $(\C,+)$}.
          J. Math. Phys. 54 (2013), no. 6, 063514, 17 pp.

\bibitem{InAb2013}  Intissar Abdelkader,
    {\it A short note on the chaoticity of a weight shift on concrete orthonormal basis associated to some Fock-Bargmann space}.
    J. Math. Phys.,  vol. 55 (2014) no. 1, pp. 011502.

\bibitem{InAb2015} Intissar A,
        {\it On the chaoticity of some tensor product weighted backward shift operators acting on some tensor product Fock-Bargmann spaces},
        Complex Analysis and Operator Theory, 2015

\bibitem{Iwaniec2002} Iwaniec H.,
        {\it Spectral methods of automorphic forms}.
        Second edition. Graduate Studies in Mathematics, 53.
        American Mathematical Society, Providence, RI; Revista Matemática Iberoamericana, Madrid, 2002.

\bibitem{Mumford83} Mumford D.,
        {\it Tata Lectures on Theta I}.
         Boston,  MA: Birkhäuser,  1983.

 \bibitem{Poincare16-54} Poincar\'e H.,
        {\it Oeuvres II, Gauthier-Villars}.
         Paris, 1916-1954 

\bibitem{PolishChuk2003} {Polishchuk A.,}
         {\it Abelian varieties, theta functions and the Fourier transform}.
         Cambridge Tracts in Mathematics, 153.
         Cambridge University Press, Cambridge, 2003.

\bibitem{Prym1882} Prym F.,
          {\it Untersuchungen uber die Riemann’sche Thetaformel und Riemann'sche Charakteristikentheorie}.
           Druck und Verlag von B. O. Teubner, Leipzig, 1882.

\bibitem{Riemann1857} Riemann G.F.B.,
        {\it Theorie der Abel'schen Functionen}.
         J. reine angew. Math. 54 (1857), pp. 115–155.

\bibitem{Ritzenthaler2004} Ritzenthaler C.,
       {\it Point counting on genus 3 non hyperelliptic curves}.
       Algorithmic number theory, 379--394, Lecture Notes in Comput. Sci., 3076, Springer, Berlin, 2004.

\bibitem{Satake1971} Satake I.,
       {\it Fock representations and theta-functions}.
       In Advances in the Theory of Riemann Surfaces. 
       Ann. of Math. Studies 66, Princeton (1971), 393-405. 

 \bibitem{Serre1973}  Serre J.-P.,
           {\it A course in arithmetic}.
           Springer-Verlag, 1973.

\bibitem{ShaskaWijesiri2008} Shaska and S. Wijesiri T.,
       {\it Codes over rings of size four, Hermitian lattices, and corresponding theta functions}.
        Proc. Amer. Math. Soc., 136 (2008), 849-960.

\bibitem{Shimura1968} Shimura Goro,
        {\it Automorphic functions and number theory}.
        {Lecture Notes in Mathematics}, No. 54 Springer-Verlag, Berlin-New York 1968.

\bibitem{Siegel1955} Siegel C.L.,
        {\it Automorphe Funktionen in mehrerer Variablen}.
         Math. Inst. Göttingen (1955)

 \bibitem{Souid2015} Souid El Ainin M.,
         {\it Concrete description of the $(\Gamma, \chi)$-theta Fock-Bargmann space for rank one in high dimension}.
          J.Complex Variables and Elliptic Equations   (2015)

\bibitem{Wirtinge1895} Wirtinger W.,
        {\it Untersuchungen über Thetafunctionen}.
         B. G. Teubner, Leipzig, 1895.

\end{thebibliography}
\end{document}